\documentclass[11pt]{amsart}
\usepackage{amsmath}
\usepackage{amssymb}
\usepackage{amsthm}
\usepackage{color}

\newtheorem{theorem}{Theorem}[section]

\newtheorem{lemma}[theorem]{Lemma}

\theoremstyle{definition}
\newtheorem{defn}[theorem]{Definition}
\theoremstyle{remark}
\newtheorem{rem}[theorem]{Remark}

\usepackage{tikz}
\usepackage{verbatim}

\tikzset{%
  >=latex, 
  inner sep=0pt,%
  outer sep=2pt,%
  mark coordinate/.style={inner sep=0pt,outer sep=0pt,minimum size=3pt,
    fill=black,circle}%
}

\usetikzlibrary{calc,fadings,decorations.pathreplacing}

\usepackage{tikz-3dplot}

\def \la{\lambda}
\renewcommand{\le}{\leqslant}

\begin{document}

\title[On conic Fourier multipliers]
{Weighted estimates for conic\\ Fourier multipliers}

\address{Departamento de Matem\'aticas -- Universidad Aut\'onoma de Madrid, 28049 Madrid, Spain}
\email{antonio.cordoba@uam.es} 
\address{Instituto de Ciencias Matem\'aticas CSIC-UAM-UC3M-UCM, 28049 Madrid, Spain}
\email{keith.rogers@icmat.es}
\thanks{Supported by the Spanish grants MTM2010-16518, MTM2011-22851, SEV-2011-0087 and the ERC grant 277778.} 
\author{Antonio C\'ordoba}
\author{Keith M. Rogers}

\maketitle

\begin{abstract} We prove a weighted inequality which  controls conic Fourier multiplier operators in terms of lacunary 
directional maximal operators. By bounding the maximal operators, this enables us to conclude that
the multiplier operators are bounded on $L^p(\mathbb{R}^3)$  with $1<p<\infty$.
\end{abstract}

\section{Introduction}

For directions $\omega\in \Omega$ contained in the unit circle
$\mathbb{S}^1\subset \mathbb{R}^2$ we form the associated
polygon~$\mathcal{P}_\Omega$, with sides contained in
$\omega+\omega^\perp$ (see Figure~\ref{fone}). We consider the associated cones
$$
\Gamma_{\!\Omega}=\Big\{ \xi\in \mathbb{R}^3\, :\, \frac{(\xi_1,\xi_2)}{|\xi_3|} \in \mathcal{P}_\Omega \Big\}
$$
and the Fourier multiplier operators $T_{\Omega}$ defined, initially on Schwartz functions, by
$$
T_{\Omega}\,:f\,\mapsto\Big(\chi_{\Gamma_{\!\Omega}} \widehat{f}\,\Big)^\vee.
$$
Here,  $\,\widehat{\,}\,$ and $^\vee$ denote the Fourier transform
and inverse transform, respectively. The purpose of this note is to
provide a condition on the directions~$\Omega$ which ensures
$T_{\Omega}$ is bounded from $L^p(\mathbb{R}^3)$ to $L^p(\mathbb{R}^3)$, where
$1<p<\infty$. The question is only interesting when $p\neq
2$, as when $p=2$ the operator is bounded by Plancherel's theorem, and no
restrictions on the directions are necessary. As shown by Fefferman \cite{F}, one cannot take all the directions of the circle, converting the polygon in a disc, when $p\neq 2$,  however one can consider the cone associated to the disc if the multiplier is smoothed out appropriately (see \cite{LV} and the references therein).

 We will control $T_{\Omega}$, via a weighted inequality, by a combination of directional
maximal operators (see the forthcoming Theorem~\ref{thetheorem2}).
This point of view has its origins in a conjecture of the first
author~\cite{C1} for the Bochner--Riesz multiplier (see also
\cite{steinwilliams,C2,Ca,CaS,Ch,LRS} for progress on this conjecture and \cite{CRF1,Ca2,B1,B2,B0} 
for closely related problems). When the directions
are lacunary of finite order (loosely speaking, a lacunary set of
order $K$ is a lacunary sequence which has lacunary sets of order
$K-1$ accumulating at each direction; see Definition~\ref{dfour}) we are able to bound the relevant maximal
operators, which yields the following theorem.

\begin{theorem}\label{thetheorem} Let $1<p<\infty$ and suppose that $\Omega\subset\mathbb{S}^1$ is lacunary of finite order. Then $T_{\Omega}$ is bounded from $L^p(\mathbb{R}^3)$ to $L^p(\mathbb{R}^3)$.
\end{theorem}

The proof of this can be found in the final section. In the
following section, we will recall a weighted inequality for singular
integrals, and prove a weighted, angular
Littlewood--Paley inequality. That angular square functions are
bounded in $L^p$ is well-known (see \cite{CRF1,NSW} for lacunary
directions and \cite{SS} for lacunary directions of finite order),
however we will require a weighted version in order to control
$T_{\Omega}$ in terms of maximal operators.

\begin{figure}\label{fone}
\begin{tikzpicture}
\newdimen\R
\R=1.5cm
           \draw[xshift=-2.0\R,fill=gray!80,opacity=0.2] 
           (0:\R)  -- (11:1.005\R) -- (22:1.01\R) -- (45:1.05\R) -- (68:1.01\R) -- (79:1.005\R) -- (90:\R) -- (101:1.005\R) -- (112:1.01\R) -- (135:1.05\R)  -- (158:1.01\R) -- (169:1.005\R) -- (180:\R) -- (191:1.005\R) -- (202:1.01\R) -- (225:1.05\R) -- (248:1.01\R) -- (259:1.005\R) -- (270:\R) -- (281:1.005\R) -- (292:1.01\R) -- (315:1.05\R) -- (338:1.01\R) -- (349:1.005\R) -- (360:\R)
              node[above] {} ;
              \draw[xshift=-2.0\R, thin,black!90, opacity=0.2, dashed] circle (\R) ;
\end{tikzpicture}
\tdplotsetmaincoords{80}{0}
\begin{tikzpicture}[tdplot_main_coords]
\def\RI{40}
\def\RII{1.8}
\draw (0,0)
 \foreach \x in {0,180,185,191,202,225,248,259,265,270,275,281,292,315,338,349,355,360} { --  (0:0) node at (0:\RI-40)  (R1-\x) {} };
\draw[dashed] (R1-0.center)
  \foreach \x in {5,11,22,45,68,79,85,90,95,101,112,135,158,169,175} { --  (0:0) node at (0:\RI-40) (R1-\x) {} };
\begin{scope}[yshift=2.5cm]
\draw[fill=gray!30,opacity=0.5] (\RII,0)
 \foreach \x in {0,5,11,22,45,68,79,85,90,95,101,112,135,158,169,175,180,185,191,202,225,248,259,265,270,275,281,292,315,338,349,355,360} { --  (\x:\RII) node at (\x:\RII) (R2-\x) {}};
\foreach \x in {0,180,225,248,259,265,270,275,281,292,315,360} { \draw (R1-\x.center)--(R2-\x.center); };
\end{scope}
\end{tikzpicture}
\caption{The polygon $\mathcal{P}_\Omega$  and part of the cone $\Gamma_{\!\Omega}$ associated to a set $\Omega$ of lacunary directions.}
\end{figure}
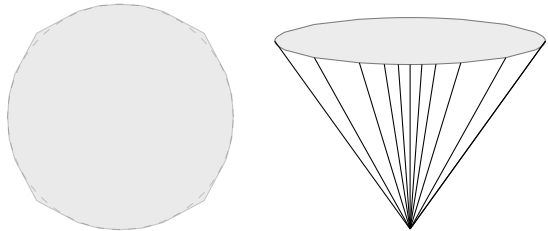

\section{An angular Littlewood--Paley inequality}\label{ang}

Theorem~\ref{thetheorem} will be proved by a  number of
applications of the following one--dimensional result for singular
integrals. There are multidimensional versions of some of the results in
this section, however we focus only on what we will need.
Fundamental to our approach is the Hardy--Littlewood maximal
operator~$M$ defined by
$$
Mf(x) =\sup_{r>0} \frac{1}{2r}\int_{-r}^r |f(x-t)|\,dt.
$$
This, and thus the composition $M^k\equiv M\circ\ldots\circ M$, is well-known to be bounded from $L^p(\mathbb{R})$ to $L^p(\mathbb{R})$ with $p>1$ (see for example the first theorem in \cite{Stein}). The following lemma was proven with $M^3 w$ replaced by $(M |w|^q)^{1/q}$, with $q>1$, in \cite{CoiF, CCF} which was then precised in  \cite{W, P}. This is sharp in the sense that the inequality does not necessarily hold if $M^3 w$ is replaced by $M^2 w$.

\begin{lemma}\label{si} Let $\mathcal{T}$ denote the singular integral operator defined by
$$\mathcal{T}f(x)=p.v.\int \mathcal{K}(x-y) f(y)\, dy,$$
where $\mathcal{K}\in C^1(\mathbb{R}\backslash\{0\})$ is such that $\mathcal{K}(x)\le c_0|x|^{-1}$ and $|\frac{d}{dx}\mathcal{K}(x)|\le c_0|x|^{-2}$. Then
$$
\int |\mathcal{T} f|^2 w\ \le \ C\!\int |f|^2 M^3w,
$$
where the constant $C$ depends only on $c_0$.
\end{lemma}

Ordering $\mathbb{S}^1$ in a clockwise sense, for each direction
$\omega\in\Omega$ we write $\omega^{\rm prev}$ for the previous
direction of $\Omega$ and write $\omega^{\rm prev}<\omega$.
Considering the regions~$A_\omega$ defined by
$$
A_\omega=\Big\{ \xi\in \mathbb{R}^3\, :\, \omega^{\rm prev}<\frac{(\xi_1,\xi_2)}{|(\xi_1,\xi_2)|}\le\omega\,\Big\},
$$
we form the associated Fourier multiplier operators $S_\omega$ defined by
$$
S_\omega f = \Big(\chi_{A_\omega} \widehat{f}\,\Big)^\vee.
$$
In the following lemma, we control the square function associated to these Fourier multipliers by two--dimensional directional maximal operators, for which we will need to recall some definitions.

The maximal operator $M_\omega$ is
 defined, initially on Schwartz functions,~by
$$
M_\omega f(x) =\sup_{r>0} \frac{1}{2r}\int_{-r}^r |f(x-t\omega)|\,dt.
$$
This is bounded from $L^p(\mathbb{R}^3)$ to $L^p(\mathbb{R}^3)$, where $p>1$, by an application of Fubini's theorem and the Hardy--Littlewood maximal theorem on the copy of the real line $\mathrm{span}(\omega)$. For a set of directions $\Omega\subset \mathbb{S}^1\subset \mathbb{R}^2$ we consider the directional maximal operator $M_\Omega$ defined by
$$
M_\Omega \,:\, f\mapsto \sup_{\omega\in\Omega}M_\omega f.
$$

We consider sequences $\{\theta_{i}\}_{i\in\mathbb{Z}}$ that satisfy
$ 0<\theta_{i+1}<\la\theta_i $ with lacunary constant $0<\la<1$. We
say that a set of directions $\Omega\subset \mathbb{S}^1$ is
lacunary (of order one) if there exists a lacunary
sequence~$\{\theta_{i}\}_{i\in\mathbb{Z}}$ so that there is at most
one direction in each of the sets
$$
\Omega_i=\Big\{ \omega \in \Omega\, :\, \theta_{i+1}<
\Big|\frac{\omega_2}{\omega_1}\Big|\le \theta_i\,\Big\}.
$$
We are also obliged to consider the perpendicular directions
$$
\Omega^\perp:=\big\{ \,\omega\times e_3 \in\mathbb{S}^1\,:\, \omega\in\Omega\,\big\},
$$
and the associated maximal operator $M_{\Omega^\perp}$. Of course, by rotational symmetry, $M_{\Omega^\perp}$ is  bounded if and only if $M_\Omega$ is bounded.

\begin{lemma}\label{onethree} Let $\Omega\subset \mathbb{S}^1$ be lacunary. Then
$$
\int_{\mathbb{R}^3} \sum_{\omega\in\Omega}|S_\omega f|^2 w\ \le \  C\! \int_{\mathbb{R}^3}  |f|^2 M^3_{e_2}M^3_{e_1}M^3_{\Omega^\perp}M^3_{e_2}M^3_{e_1} w,
$$
where $C$ depends only on the lacunary constant $\la$.
Moreover,
$$
\int_{\mathbb{R}^3} |f|^2 w\ \le \  C\!\int_{\mathbb{R}^3}   \sum_{\omega\in\Omega}|S_\omega f|^2 M^3_{e_2}M^3_{e_1}M^3_{\Omega^\perp}M^3_{e_2}M^3_{e_1} w.
$$
\end{lemma}

\begin{proof} Let $\phi=\phi_{\rm o}\ast \phi_{\rm o}$ be a Schwartz function, supported on $[-\la^{-1/8}, -\la^{{1/8}}]\cup[\la^{{1/8}},\la^{-1/8}],$ such that
\begin{equation}\label{part}
\sum_{i\in \mathbb{Z}} \phi^2(\la^{i/8}\xi)=1,\quad \xi\in\mathbb{R}\backslash\{0\},
\end{equation}
where $\la$ is the constant of lacunarity of the directions $\Omega$.
Consider the Fourier projection operators $P_{ij}$ onto dyadic rectangles defined by
$$\widehat{P_{ij} f}(\xi)=\phi(\la^{i/8}\xi_1)\phi(\la^{j/8}\xi_2)\widehat{f}(\xi).$$
First we will prove that there is a constant $C$, depending only on $\la$, such that
\begin{align}\label{fr}
\int_{\mathbb{R}^3} \sum_{i,j\in\mathbb{Z}} |P_{ij}f|^2 w\ \le \  C\! \int_{\mathbb{R}^3}  |f|^2 M^3_{e_2}M^3_{e_1} w\\
\int_{\mathbb{R}^3}|f|^2 w\ \le \  C\! \int_{\mathbb{R}^3}   \sum_{i,j\in\mathbb{Z}} |P_{ij}f|^2M^3_{e_2}M^3_{e_1} w.\nonumber
\end{align}
This will follow by factorising into compositions of
one--dimensional operators $P_{ij}=P^2_jP^1_i$, in the obvious way,
and applying Lemma~\ref{si}. These inequalities are well-known; see for example \cite{B2}. We prove them directly using Lemma~\ref{si} for the convenience of the reader.

Indeed, for all $t\in[0,1]$, we first apply  Lemma~\ref{si}  to the
operators $$f\mapsto\sum_{i\in\mathbb{Z}}r_i(t)P_{i}^1f,$$ where
$r_i(t)=r_0(2^i t)$ and $r_0$ is the Rademacher function, equal to
one on $[k,k+1/2)$ and minus one on $[k-1/2,k)$ for all $k\in
\mathbb{Z}$. To see that the decay conditions of the lemma are
satisfied, uniformly in $t\in [0,1]$, we note that the associated
kernel $\mathcal{K}$ satisfies
$$
\mathcal{K}(x)=\sum_{i\in\mathbb{Z}}
r_i(t)\la^{-i/8}\phi^\vee(\la^{-i/8}x)\ \approx \sum_{i\le 8\log_{\la^{-1}}
|x|^{-1}} r_i(t)\la^{-i/8}\phi^\vee(\la^{-i/8}x),
$$
which follows by using the rapid decay of the Schwartz function $\phi^\vee$.
We obtain that
$$
\int_{\mathbb{R}^3}  \Big|\sum_{i\in\mathbb{Z}}r_i(t)P_{i}^1g\Big|^2 w \ \le \  C\!\int_{\mathbb{R}^3}  |g|^2 M^3_{e_1}w,
$$
where the constant is independent of $t\in[0,1]$.
By taking $g=\sum_{j} r_j(s) P^2_{j}f$, in particular this implies that
$$
\int_{\mathbb{R}^3}  \Big|\sum_{i,j\in\mathbb{Z}}r_i(t)r_j(s)P_{ij}f\Big|^2 w\ \le \  C\!\int_{\mathbb{R}^3}  \Big|\sum_{j\in\mathbb{Z}}r_j(s)P^2_{j}f\Big|^2 M^3_{e_1}w,
$$
where the constant $C$ is independent of both $t\in[0,1]$ and $s\in[0,1]$.
Applying Lemma~\ref{si} again, this time in the $x_2$ variable, yields
\begin{align}\label{pold}
\int_{\mathbb{R}^3}  \Big|\sum_{i,j\in\mathbb{Z}}r_i(t)r_j(s)P_{ij}f\Big|^2 w \ \le \  C\! \int_{\mathbb{R}^3}  |f|^2 M^3_{e_2}M^3_{e_1}w.
\end{align}
Integrating over $[0,1]\times[0,1]$ with respect to $(t,s)$ yields the first inequality of \eqref{fr} by the orthogonality of the Rademacher functions.

To see the reverse inequality, we take $f=\sum_{k,\ell}r_k(t)r_\ell(s)P_{k\ell}g$ and note that
\begin{align*}
&\sum_{i,j\in\mathbb{Z}}r_i(t)r_j(s)P_{ij}\Big[\sum_{k,\ell\in\mathbb{Z}}r_k(t)r_\ell(s)P_{k\ell}g\Big]\\
=\ &\sum_{i,j\in\mathbb{Z}}r^2_i(t)r^2_j(s) P_{ij}P_{ij} g\ +\sum_{i,j,k,\ell\in\mathbb{Z} : i\neq k,j\neq \ell}r_i(t)r_j(s)r_{k}(t)r_{\ell}(s) P_{ij}P_{k\ell} g\\
=\ & \ g\ +\sum_{i,j,k,\ell\in\mathbb{Z} : i\neq k,j\neq \ell}r_i(t)r_j(s)r_{k}(t)r_{\ell}(s) P_{ij}P_{k\ell} g,
\end{align*}
using \eqref{part} and that $r^2_ir^2_j\equiv1$.  Substituting into
\eqref{pold} and integrating over $[0,1]\times[0,1]$, we obtain 
$$
\int_{\mathbb{R}^3}\Big(|g|^2+\sum_{i,j\in\mathbb{Z}: i\neq j}|P_{ij}P_{ji}g|^2 \Big)w\ \le \  C\! \int_{\mathbb{R}^3}   \sum_{k,\ell\in\mathbb{Z}} |P_{k\ell}g|^2M^3_{e_2}M^3_{e_1}w.
$$
by the orthogonality of the Rademacher functions as before. Note that unless $j=i\pm 1$, the sumands in the first sum are identically zero, but in any case they are positive and so this implies 
the second inequality of \eqref{fr}, for positive weights $w$. Then the inequality for positive weights implies the inequality for general weights.

Now our desired angular inequalities will follow by combining these more traditional Littlewood--Paley inequalities with a further application of Lemma~\ref{si} to the operator $H_\omega$ defined by
$$
\widehat{H_{\omega} f}(\xi) = \chi_{\omega\cdot \xi \le  0} \widehat{f}(\xi).
$$
Note that $H_{\omega}=\frac{1}{2}({\rm Id}+\frac{1}{i} \mathcal{H}_\omega)$, where $\mathcal{H}_{\omega}$ denotes the directional Hilbert transform defined by
$$
\mathcal{H}_{\omega} f(x)= \frac{1}{\pi}\, p.v.\!\int f(x-\omega t)\, \frac{dt}{t}.
$$
In the copy of the real line  $\mathrm{span}(\omega)$, this  clearly satisfies the conditions of the lemma.

\begin{figure}\label{fthree}
\centering
\[\begin{picture}(128,128)
\put(0,0){\line(0,1){128}} \put(0,0){\line(1,0){128}}
\put(8,0){\line(0,1){128}} \put(0,8){\line(1,0){128}}
\put(16,0){\line(0,1){128}} \put(0,16){\line(1,0){128}}
\put(32,0){\line(0,1){128}} \put(0,32){\line(1,0){128}}
\put(64,0){\line(0,1){128}} \put(0,64){\line(1,0){128}}
\put(128,0){\line(0,1){128}} \put(0,128){\line(1,0){128}}
\put(90.51,0){\line(0,1){128}} \put(0,90.51){\line(1,0){128}}
\put(45.255,0){\line(0,1){128}} \put(0,45.255){\line(1,0){128}}
\put(22.63,0){\line(0,1){128}} \put(0,22.63){\line(1,0){128}}
\put(11.3,0){\line(0,1){128}} \put(0,11.3){\line(1,0){128}}
\put(0,0){\line(5,6){106.5}}
\put(0,0){\line(1,3){42.6666}}
\put(0,0){\line(2,1){128}}
\put(0,0){\line(1,6){21.33333}}
\put(0,0){\line(5,1){128}}
\end{picture}\]
\caption{Segments with lacunary constant $1/2$ intersecting
rectangles with lacunary constant $1/\sqrt{2}$.}
 \end{figure}
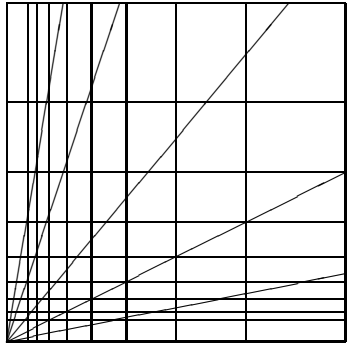

By enumerating the directions $\omega$ of $\Omega$ and writing $r_\omega$ for the associated Rademacher function we apply the inequalities of \eqref{fr};
\begin{align}\label{ror}
\int_{\mathbb{R}^3}  \Big|\sum_{\omega\in\Omega} r_\omega(t)S_\omega f\Big|^2 w
&\le \ C\!\int_{\mathbb{R}^3} \sum_{i,j\in\mathbb{Z}} \Big|\sum_{\omega: P_{ij}S_\omega\neq 0}r_\omega(t)S_\omega P_{ij}f\Big|^2 M^3_{e_2}M^3_{e_1} w\nonumber\\
&\le \ C\!\sum_{i,j\in\mathbb{Z}} \int_{\mathbb{R}^3} |S_\omega P_{ij}f|^2 M^3_{e_2}M^3_{e_1} w\\
&\le\ C\!\int_{\mathbb{R}^3}  \sum_{i,j\in\mathbb{Z}}|P_{ij}f|^2 M^3_{\Omega^\perp}M^3_{e_2}M^3_{e_1}w\nonumber\\
&\le \ C\!\int_{\mathbb{R}^3} |f|^2 M^3_{e_2}M^3_{e_1}M^3_{\Omega^\perp}M^3_{e_2}M^3_{e_1}w,\nonumber
\end{align}
uniformly in $t\in[0,1]$.
Here, the second inequality is due to the fact that there are at
most two directions $\omega$ which satisfy $P_{ij}S_\omega\neq 0$
for a given $i$ and~$j$ (see Figure~\ref{fthree}). The third inequality is by
Lemma~\ref{si} in the directions $\pm\omega\times e_3$ or
$\pm\omega^{\rm prev}\times e_3$ as $S_{\omega}P_{ij}$ is equal to
either $H_{\pm\omega\times e_3}P_{ij}$ or $H_{\pm\omega^{\rm
prev}\times e_3}P_{ij}$ (if it is not simply $P_{ij}$). Integrating
over $[0,1]$ with respect to $t$ yields the first of our desired
inequalities.

To see the reverse inequality, we take $f=\sum_{\omega'}r_{\omega'}(t)S_{\omega'} g$ as before, so that
$$
\sum_{\omega\in\Omega} r_\omega(t)S_\omega\sum_{\omega'\in\Omega}r_{\omega'}(t)S_{\omega'} g=\sum_{\omega\in\Omega} r^2_\omega(t) S_\omega g=g,
$$
as $r^2_\omega\equiv 1$.
Substituting into \eqref{ror} we see that
$$
\int_{\mathbb{R}^3}  |g|^2 w\ \le \  C\!\int_{\mathbb{R}^3} \Big|\sum_{\omega'\in\Omega}r_{\omega'}(t)S_{\omega'} g\Big|^2 M^3_{e_2}M^3_{e_1}M^3_{\Omega^\perp}M^3_{e_2}M^3_{e_1}w,
$$
and so by integrating over $[0,1]$ with respect to $t$, we obtain
the second inequality.
\end{proof}

\section{Proof of Theorem~\ref{thetheorem}}\label{four}

First of all,  we give a precise version of the definition of finite order lacunary sets as mentioned in the introduction.

\begin{defn}\label{dfour} If $\Omega$ consists of a
single direction we say that it is {\it lacunary of order zero}. We then say that $\Omega\subset \mathbb{S}^1\subset\mathbb{R}^2$ is {\it lacunary of order $K$} if there is a choice of basis $\{e_1,e_2\}$ and lacunary sequence $\{\theta_i\}_{i\in \mathbb{Z}}$ such that the sets
$$
\Omega_i=\Big\{ \omega \in \Omega\, :\, \theta_{i+1}<
\Big|\frac{\omega_2}{\omega_1}\Big|\le \theta_i\,\Big\}.
$$
are lacunary of order $\le K-1$ for all $i\in \mathbb{Z}$, with
uniformly bounded lacunary constants.
\end{defn}

 Sj\"ogren and
Sj\"olin \cite{SS} proved that $M_{\Omega}$ is bounded from
$L^p(\mathbb{R}^2)$ to $L^p(\mathbb{R}^2)$, where $p>1$, if $\Omega$ is lacunary of
finite order (see also \cite{CRF2,NSW} for precedents and \cite{B}
for an equivalence).

\begin{lemma}\label{pwq} If $\Omega\subset \mathbb{S}^1\subset \mathbb{R}^2$ is lacunary of finite order and
$$
\Omega^{\rm 3d}:=\Big\{ \,\tfrac{1}{\sqrt{2}}(\omega\pm e_3)\, :\, \omega\in\Omega\,\Big\},
$$
then $M_{\Omega^{\rm 3d}}$ is bounded from $L^p(\mathbb{R}^3)$ to $L^p(\mathbb{R}^3)$ with $p>1$.
\end{lemma}

\begin{proof} We employ the localisation principle due to Parcet and the second author~\cite{PR}.
For directions $\Omega\subset \mathbb{S}^2\subset \mathbb{R}^3$ and a lacunary sequence $\{\theta_i\}_{i\in\mathbb{Z}}$ we define segments $\Omega_{i}^\ell$ by
$$
\Omega_{i}^\ell=\Big\{ \omega \in \Omega\, :\, \theta_{i+1}<
\Big|\frac{\omega_k}{\omega_j}\Big|\le \theta_i\,\Big\},\qquad j<k,
$$
where $e_\ell=e_j\times e_k$ or $-e_j\times e_k$. Then
\begin{equation}\label{locp}
\|M_\Omega\|_{p\to p}\ \le \ C \sup_{\ell=1,2,3}\sup_{i\in\mathbb{Z}}\|M_{\Omega_{i}^\ell}\|_{p\to p},
\end{equation}
where $C$ depends only on $p>1$ and the lacunary constant (see also \cite{A} for a two-dimensional version). With this we can separate
the directions of $\Omega^{\rm 3d}$ into isolated directions. The
main difficulty is to choose the basis correctly between each
application so that the lacunarity of $\Omega^{\rm 3d}$ is \lq
visible' in each direction.

Initially we change the basis vector $e_3$ by $e_3\to
\tfrac{1}{\sqrt{2}}(\omega_{0}\pm e_3)$, where
$\omega_{0}\in\Omega\subset \mathbb{S}^1$ is one of the accumulation
points  of highest order, and take $e_1$ so that $e_3+e_1$ is
tangent to the circle which contains~$\Omega^{\rm 3d}$.
 This also fixes the final basis vector $e_2$, chosen so that $e_3=e_1\times e_2$.
 Now, if~$\Omega$ is lacunary of order~$K$ with lacunary constant $\la$,
 we can take $\theta_i=\la^{i/2}$ for all $i\in\mathbb{Z}$. Then the directions of~$\Omega$ which
 correspond to the
 directions of the resulting segments
  of the localisation principle are lacunary of order $\le K-1$ (at
 least after eight applications of~\eqref{locp}), with lacunary constants uniformly bounded by $\la$. Thus,
  after a finite number of applications of~\eqref{locp}, each time changing the basis as before to include one of the accumulation points of highest order, we separate the directions into isolated directions, and so  $M_{\Omega^{\rm 3d}}$ is bounded.
\end{proof}

Now as $T_{\Omega}$ is self-adjoint, the boundedness in $L^p$ follows from the boundedness in $L^{p'}$, where $\frac{1}{p}+\frac{1}{p'}=1$. To see that $T_{\Omega}$ is bounded in $L^p$ with $2< p<\infty$, by H\"older's inequality and the following weighted inequality it suffices to bound the maximal operators from~$L^{(p/2)'}$ to $L^{(p/2)'}$. Thus, Theorem~\ref{thetheorem} is obtained by combining the following theorem with the previous lemma and the boundedness of $M_{\Omega\cup\Omega^\perp}$ due to Sj\"ogren--Sj\"olin \cite{SS}.

\begin{theorem}\label{thetheorem2} Let $\Omega\subset \mathbb{S}^1$ be lacunary of order~$K$. Then
$$
\int_{\mathbb{R}^3} |T_{\Omega}f|^2 w\ \le \  C\! \int_{\mathbb{R}^3}  |f|^2 M^{15K}_{\Omega\cup\Omega^\perp}M^6_{\Omega^{\rm 3d}}M^{15K}_{\Omega\cup\Omega^\perp}w,
$$
where $C$ depends only on the lacunary constant $\la$ and the lacunary order $K$.
\end{theorem}

\begin{proof} Noting that $M_{\Omega}=M_{\overline{\Omega}}$ we see that the maximal operators $M_{e_1}$ and $M_{e_2}$,  with $e_1=\omega_0$ and $e_2=\omega_0^\perp$, where $\omega_0$ is an accumulation point, are dominated by $M_{\Omega\cap\Omega^\perp}$. Thus, by $K$ applications of the second inequality of Lemma~\ref{onethree}, we see that
\begin{align*}
\int_{\mathbb{R}^3} |T_{\Omega}f|^2 w\ \le \  C\!\int_{\mathbb{R}^3} \sum_{\omega\in\Omega}|S_\omega T_{\Omega}f|^2 M^{15K}_{\Omega\cup\Omega^\perp} w.
\end{align*}
Now we can write
$$
S_\omega T_{\Omega}f=H_{\frac{1}{\sqrt{2}}(\omega-e_3)} H_{\frac{1}{\sqrt{2}}(\omega^{\rm prev}-e_3)}S_\omega f+H_{\frac{1}{\sqrt{2}}(\omega+e_3)} H_{\frac{1}{\sqrt{2}}(\omega^{\rm prev}+e_3)}S_\omega f,
$$ so that, by two more applications of Lemma~\ref{si} in one of these four directions, we see that
\begin{align*}
\int_{\mathbb{R}^3} |T_{\Omega}f|^2 w&\le C\!\sum_{\omega\in\Omega}\int_{\mathbb{R}^3} |S_\omega f|^2 M^6_{\Omega^{\rm 3d}}M^{15K}_{\Omega\cup\Omega^\perp}w\\
&\le C\!\int_{\mathbb{R}^3} |f|^2
M^{15K}_{\Omega\cup\Omega^\perp}M^6_{\Omega^{\rm
3d}}M^{15K}_{\Omega\cup\Omega^\perp}w,
\end{align*}
where the second inequality is by $K$ applications of the first inequality of Lemma~\ref{onethree}. This completes the proof.
\end{proof}

\begin{rem} If we consider the two--dimensional multiplier operator $R_{\Omega}$ associated to a polygon $\mathcal{P}_\Omega$, defined by
$$
R_{\Omega}\,:f\,\mapsto\Big(\chi_{\mathcal{P}_\Omega} \widehat{f}\,\Big)^\vee,
$$
where $\Omega\subset \mathbb{S}^1$ is lacunary of order $K$, the above argument also yields the weighted inequality
$$
\int_{\mathbb{R}^2} |R_{\Omega}f|^2 w\ \le \  C\! \int_{\mathbb{R}^2}  |f|^2 M^{30K+6}_{\Omega\cup\Omega^\perp}w,
$$
where the constant $C$ depends only on the lacunary constant $\la$ and the lacunary order $K$.
\end{rem}

\begin{rem} In higher dimensions, we are able to consider polytopes associated to lacunary directions with an exact product structure. That is to say that the orthogonal projections of the directions onto the two--dimensional subspaces, formed by the span of two basis vectors, are lacunary. For example, if we consider the $n$--dimensional multiplier operator
$$
R_{\Omega}\,:f\,\mapsto\Big(\chi_{\mathcal{P}_\Omega} \widehat{f}\,\Big)^\vee,
$$
where $\Omega\subset \mathbb{S}^{n-1}$ consists of normalised versions of the directions $$\{(2^{k_1},\ldots,2^{k_n})\}_{k_1,\ldots,k_n\in\mathbb{Z}},$$ the above argument also yields the weighted inequality
$$
\int_{\mathbb{R}^n} |R_{\Omega}f|^2 w\ \le \  C\! \int_{\mathbb{R}^n}  |f|^2 M^{30(n-1)+3\cdot2^{n-1}}_{\Omega}\!\!w,
$$
where the constant $C$ depends only on the dimension. For this we first use $n-1$ applications of the second inequality of Lemma~\ref{onethree}  (with $\mathbb{R}^3$ replaced by $\mathbb{R}^n$ and the lacunary sectors, divided by powers of $2$, taken in $\mathrm{span}(e_j,e_{j+1})$ with $j=1,\ldots,n-1$). Note that in this case the maximal operators $M_{e_j}$ are bounded by $M_\Omega$. This separates the multiplier in such a way that we are left to deal with the composition of $2^{n-1}$ multiplier operators of the type that can be dealt with by Lemma~\ref{si} applied in one of the directions of $\Omega$. We then recompose the operator with $n-1$ applications of the second inequality of Lemma~\ref{onethree}. This maximal operator was bounded by Carbery~\cite{Ca2}, and so the multiplier associated to these polytopes with infinite faces are bounded on  $L^p(\mathbb{R}^n)$ with $1<p<\infty$.
\end{rem}

\vspace{1em}

The authors thank Jon Bennett for bringing \cite{P} and \cite{W} to their attention, and Michael Bateman and Jon Bennett for independently pointing out a minor error.

\end{document}